\documentclass[a4paper, 12pt]{article}

\usepackage[utf8]{inputenc}
\usepackage{graphicx}
\usepackage{hyperref}

\usepackage[backend=bibtex]{biblatex}
\DeclareGraphicsExtensions{.pdf}

\usepackage[top=1.5cm, bottom=2cm, left=2cm, right=2cm]{geometry}

\usepackage{pgf,tikz,pgfplots}
\pgfplotsset{compat=1.15}
\usepackage{mathrsfs}
\usetikzlibrary{arrows}

\usepackage{lmodern} 
\usepackage{amsmath}
\usepackage{amssymb}
\usepackage{amsthm}

\usepackage{xcolor}

\usepackage{hyperref}
\usepackage{cleveref}

\newcommand{\defn}[1]{{ \bf#1}}

\newtheorem{thm}{Theorem}
\newtheorem{lem}[thm]{Lemma}
\newtheorem{prop}[thm]{Proposition}
\newtheorem{defi}[thm]{Definition}

\newtheorem{ex}[thm]{Example}

\newtheorem{remark}[thm]{Remark}

\def\PQSym{\textbf{PQSym}}
\def\PQSymstar{\textbf{PQSym}^*}
\def\G{{\bf G}}
\def\F{{\bf F}}

\DeclareSymbolFont{Shuffle}{U}{shuffle}{m}{n}
\DeclareMathSymbol\shuffle{\mathbin}{Shuffle}{"001}

\title{The free and parking quasi-symmetrizing actions}
\author{Adrien Segovia\thanks{Université du Québec à Montréal, LACIM}}
\date{}

\begin{document}

\maketitle

\abstract{We define two actions of the infinite symmetric group on the set of words on positive integers, called the free and parking quasi-symmetrizing actions, whose invariants are respectively the elements of the Hopf algebras $\textbf{FQSym}^*$ and $\textbf{PQSym}^*$. We study in depth the parking quasi-symmetrizing action by generalizing it to actions with a parameter $r\in(\mathbb{N}\setminus \{0\} )\bigcup\{\infty\}$. We prove that the spaces of the invariants under these $r$-actions form an infinite chain of nested graded Hopf subalgebras of $\textbf{PQSym}^*$. We give some properties of these Hopf algebras including their Hilbert series, a basis, and formulas for their product and coproduct. Finally we look more closely at the case $r=\infty$, obtaining enumerative results related to trees with maximal decreasing subtrees of given sizes.
}

\section*{Introduction}

In algebraic combinatorics, the most well-known algebra is maybe $Sym$, the Hopf algebra of symmetric functions. By definition, its elements are the invariants under a symmetric group action, the one that permutes the variables. Hopf algebras are in general not defined as the invariants under an action, but we know of a few cases where such a symmetric action was defined to recover the elements of an algebra that was not previously defined in this way; for the algebra $\textbf{NCSF}$ of non-commutative symmetric functions it was done by Lascoux and Schützenberger \cite{LascouxSch} and for the algebras $QSym$, $WQSym$ and $WSym$ of quasi-symmetric, word quasi-symmetric and word symmetric functions it was done by Hivert \cite{hivert1999combinatoire}. In this article, we define the free and parking quasi-symmetrizing actions which enable us to give equivalent definitions of the Hopf alegbras $\textbf{FQSym}^*$ and $\PQSymstar$ of respectively free and parking quasi-symmetric functions.

$\textbf{FQSym}$, which is also called the Malvenuto-Reutenauer algebra \cite{MalReu}, is a Hopf algebra on permutations. It is a self-dual Hopf algebra and its polynomial realization is defined using an algorithm on words, the standardization, sending words to permutations. A basis element indexed by a permutation $\sigma$ is set to be equal to the sum of all words on positive integers whose standardization is $\sigma$.
One of our results, Theorem~\ref{thmFQSym}, states that the polynomials realizing $\textbf{FQSym}^*$ are the invariants under a symmetric group action, which we call the free quasi-symmetrizing action.
By using, in the polynomial realization of $\textbf{FQSym}$, commutative variables instead of noncommutative ones, we obtain the Hopf algebra QSym of quasi-symmetric functions. This algebra and a noncommutative analogue $\textbf{WQSym}$ of word quasi-symmetric functions have been well-studied \cite{hivert1999combinatoire,foissy2007bidendriform,bergeron2009hopf}. Hivert defined the quasi-symmetrizing action \cite{hivert1999combinatoire}, an action of the symmetric group on formal power series whose invariants are the polynomials realizing $QSym$. He did the same for $\textbf{WQSym}$ by defining the word quasi-symmetrizing action, a symmetric group action on words. By generalizing his actions with a parameter $r\in\mathcal{N}:=\big(\mathbb{N}\setminus \{0\}\big)\bigcup \{\infty\}$, he obtained \cite{hivertlocal} that the spaces of the invariants under these $r$-actions form an infinite chain of nested Hopf subalgebras interpolating between $QSym$ and $Sym$, the case $r=1$ being $QSym$ and the case $r=\infty$ being the Hopf algebra of symmetric functions $Sym$, and in the noncommutative analogue an interpolation between $\textbf{WQSym}$ and $\textbf{WSym}$, the Hopf algebra of word symmetric functions. Using the interpolation between $QSym$ and $Sym$, Garsia and Wallach \cite{GARSIA2007704} gave a much simpler proof of their previous result that the algebra of quasi-symmetric polynomials is free as a module over the ring of symmetric polynomials.

A\defn{parking function} of length $n$ is a word $u=u_1\dots u_n$ on positive integers whose nondecreasing rearrangement $u^{\uparrow}=u_1'\dots u_n'$ satisfies $u_i'\leq i$ for all $i$. The sets of parking functions of length $n\leq 3$ are
$\{\epsilon\}$, $\{1\}$, $\{11, 12, 21\}$, and
\vspace{-2mm}
$$\{111, 112, 113, 121, 122, 123, 131, 132, 211, 212, 213, 221, 231, 311, 312, 321\}.$$

\vspace{-2mm}
\noindent There are $(n+1)^{n-1}$ parking functions of length $n$ and $\frac{1}{n+1} \binom{2n}{n}$ nondecreasing ones.
Parking functions have been first studied to solve problems about data storage in computer science \cite{KonheimWeiss} and since then, they appeared in various areas of combinatorics (see \cite{stanley1998hyperplane} and the references given there). 
We can see parking functions as a generalization of permutations where some repeated letters are allowed.

In 2003, Novelli and Thibon defined a Hopf algebra on parking functions, denoted by $\PQSym$ for parking quasi-symmetric functions \cite{articlePQSym,novelli2005construction}. Its dual algebra $\PQSymstar$ can be realized as a subalgebra of $\mathbb{K}\langle A\rangle$, the free associative algebra over a totally ordered infinite alphabet A with $\mathbb{K}$ a field of characteristic $0$, and has $\textbf{FQSym}$ and $\textbf{WQSym}$ as Hopf subalgebras. A polynomial realization of the Hopf algebra $\PQSymstar$ is defined using an algorithm on words, the parkization, sending words to parking functions. A basis element indexed by a parking function $u$ is set to be equal to the sum of all words on positive integers whose parkization is $u$.
One of our results, Theorem~\ref{thmpolynomialrealization}, states that the elements of $\PQSymstar$ are the invariants under a symmetric group action, called the parking quasi-symmetrizing action, the basis elements being sums on the orbits. Similarly to \cite{hivertlocal}, we also generalize this action to an action with a parameter $r\in\mathcal{N}$ and obtain that the subspaces of the invariants form an infinite chain of nested graded Hopf subalgebras of $\PQSymstar$ (Theorem~\ref{thminterpolation}). We give properties of these Hopf subalgebras (Hilbert series, basis, formulas for the product and coproduct). The case $r=\infty$ is of particular interest since the dimensions of the homogeneous components of the Hopf algebra are related to trees with maximal decreasing subtrees of given sizes (Theorem~\ref{thdimtrees}). As a byproduct, we obtain a new point of view on the enumeration of these trees.

In Section~\ref{secBack} we recall the definitions of the Hopf algebras $\textbf{FQSym}^*$ and $\PQSymstar$, and recall some results of Hivert that inspired this article. In Section~\ref{actionperm} we introduce the free quasi-symmetrizing action.
In Section~\ref{sectionparkingquasi} we introduce the parking quasi-symmetrizing action. In Section~\ref{sectionactionparkingrR} we generalize this last action and introduce the $r$-parking quasi-symmetrizing actions, and in Section~\ref{sectioncasinfini} we focus on the case $r=\infty$.

\section{Background}
\label{secBack}

\subsection{The Hopf algebras $\textbf{FQSym}^*$ and $\PQSym^*$}
\label{rappelHopf}

Let $w$ be a word on positive intergers. There is an algorithm, called the\defn{standardization}, that maps a word $w$ to a permutation denoted by $\mathrm{Std}(w)$. It is defined as follows; we iteratively scan from left to right the word $w$, replace the occurences of the smallest letter with $1,2,\dots, k$, then the occurences of the second smallest letter with $k+1, k+2,\dots$ and so on. For example, $\mathrm{Std}(823278)=513246$.

The\defn{shuffle} $\shuffle$ is an associative and commutative operation on the linear span of the set of words on an alphabet $A$, with coefficients in a field $\mathbb{K}$ of characteristic $0$. The empty word is its identity element and if $a,b$ are letters and $u,v$ are words, we have $(au)\,\shuffle \,(bv) = a(u\,\shuffle \, bv) + b(au\,\shuffle \,v)$.
The\defn{shifted shuffle} $\Cup$ of two words $u,v$ is defined by 
$u \Cup v := u\shuffle v[|u|]$, where $v[|u|]$ is the word $v$ whose letters have all been incremented by the length of $u$. For example, we have
\begin{align*}
    12\,\shuffle \,21 &=  2\times 1221 + 1212 + 2121 + 2\times 2112 \\
    12\,\Cup \,21 &= 12\,\shuffle \,43 = 1243 + 1423 + 4123 + 1432 + 4132 + 4312 .
\end{align*}

\noindent Since the shifted shuffle of permutations
is a multiplicity-free sum, we will identify it with its underlying set. Although we will not use in this article the structure of Hopf algebra of $\textbf{FQSym}$ (only its vector space structure), let us define it.

\begin{defi}[\cite{MalReu}]
The Hopf algebra $\defn{\textbf{FQSym}}$ of\defn{free quasi-symmetric functions}, also called the Malvenuto-Reutenauer algebra, whose bases are indexed by permutations, has a basis $(\F_u)$ in which the product and coproduct are defined by
\begin{equation*}
\F_{\sigma} \F_{\tau} = \sum_{w\,;\, \,w \in \sigma \, \Cup \, \tau}  \F_w    
\end{equation*}
and 
\begin{equation*}
\Delta \F_{\sigma} = \sum_{u,v \,;\, u\cdot v ={\sigma}} \F_{\mathrm{Std}(u)} \otimes \F_{\mathrm{Std}(v)}    
\end{equation*}
 where $\cdot$ is the concatenation of words.
\end{defi}

The Hopf algebra $\textbf{FQSym}$ is noncommutative and noncocommutative. It is naturally graded by the size of permutations and is connected (the dimension of the homogeneous component of degree zero is $1$). Its graded dual is then also a Hopf algebra, denoted by $\textbf{FQSym}^*$. This dual is in fact isomorphic to $\textbf{FQSym}$. Let $\G_{\sigma} := \F_{\sigma}^*$ be the basis elements of the dual basis of $\textbf{FQSym}$. The algebra $\textbf{FQSym}^*$ is generated as a vector space by the set of the 

\begin{equation}
    \label{polyrealfq}
    \displaystyle \G_{\sigma}=\sum_{w\,; \,\mathrm{Std}(w)=\sigma} w, 
\end{equation}

\noindent for all $\sigma \in \bigcup_{n\geq 0} \mathfrak{S}_n$. For example, we have $\displaystyle \G_{231} = \sum \limits_{1 \leq  a < b\leq c} bca $.

Let us now focus on parking functions. Let $PF_n$ be the set of parking functions of length $n$ and $PF:=\bigcup_{n\geq 0} PF_n$ be the set of all parking functions.  A\defn{prime parking function} is a parking function $u$ whose nondecreasing rearrangement $u^{\uparrow}=u_1'\dots u_n'$ satisfies that $u_i'< i$ for all $i\geq 2$. Let $PPF_n$ be the set of prime parking functions of length $n$. The prime parking functions of length $n\leq 3$ are $\{\epsilon\}$, $\{1\}$, $\{11\}$, and $\{111, 112, 121, 211\}$.

\begin{prop}[\cite{Stanley_2023}, Exercise $49.(f)$]
\label{cardinalparkingfunctions}
We have, for all $n> 1$, $|PPF_n|=(n-1)^{n-1}$.
\end{prop}

If $u\in PF_n$, we decompose its nondecreasing rearrangement $u^{\uparrow}=u_1'\dots u_n'$ by putting vertical bars just before the letters $u_i'\neq 1$ that satisfy $u_i'=i$. For example, if $u~=~25461851$, we have $u^{\uparrow}=112|4|556|8$. The\defn{prime blocks} of $u^{\uparrow}=112|4|556|8$ are the four subwords $112$, $4$, $556$, and $8$. The prime blocks of $u$ are obtained from those of $u^{\uparrow}$ by taking the subwords of $u$ on the letters of those words. The subwords of $u=25461851$ on the sets of letters $\{1,2\}$, $\{4\}$, $\{5,6\}$, and $\{8\}$ are the four prime blocks $211$, $4$, $565$, and~$8$.

The\defn{parkization} procedure \cite{articlePQSym} is an operation mapping words on positive integers to parking functions.  Let $w=w_1\dots w_n$ be a word on positive integers. The parkized word of $w$, denoted by $\mathrm{Park}(w)$, is obtained by the following recursive algorithm.

\noindent Let $d(w):= \min \{i \,: \,\, |\{j,\,w_j\leq i\} | < i \}$.
\begin{itemize}
\item[$\bullet$]  If $d(w)=n+1$, then return $w$.
\item[$\bullet$]  Otherwise, let $w'$ be the word obtained from $w$ by decrementing by one all letters greater than $d(w)$. Then return the parkized word of $w'$.
\end{itemize}

\begin{ex}
\label{exparkization}
If we start with $w=83493$, we obtain

\begin{tikzpicture}
\draw (0,0) node{$83493$};
\draw (1,0) node{$\longrightarrow$};
\draw (2,0) node{$72382$};
\draw (3,0) node{$\longrightarrow$};
\draw (4,0) node{$61271$};
\draw (5,0) node{$\longrightarrow$};
\draw (6,0) node{$51261$};
\draw (7,0) node{$\longrightarrow$};
\draw (9.5,0) node{$41251=\mathrm{Park}(83493) .$};
\end{tikzpicture}

\vspace{-1.9mm}
Indeed, we have $d(83493)=1$, so all letters greater than $1$ are decremented by one. We obtain  $72382$, and since $d(72382)=1$, we do the same, obtaining $61271$. We then have $d(61271)=4$, so all letters greater than $4$ are decremented by one, {\em etc.}
\end{ex}

This algorithm terminates since $w' < w$ for the lexicographic order. It returns a parking function since a word of length $n$ on positive integers is a parking function if, and only if the number of letters smaller than $i$ is greater than or equal to $i$, for all $1\leq i\leq n$.

In Example \ref{exparkization}, the prime blocks of $41251=\mathrm{Park}(83493)$ are $121$, $4$, and $5$. The prime blocks respectively came from the letters $343$, $8$, and $9$ in $83493$. An important observation that we will use in Section \ref{sectionparkingquasi} (Lemma~\ref{lemmaparkization}) is that the letters at the same positions in $121$ and $343$ have always the same difference, equal to $2$. Also the sequence of the differences $2$, $4=8-4$, and $4=9-5$ between the prime blocks of $\mathrm{Park}(83493)$ and the letters at the same positions in $83493$ is a nondecreasing sequence $(2,4,4)$.

We now define the Hopf algebra $\PQSym$. Since the shifted shuffle of parking functions is a multiplicity-free sum, we will identify it with its underlying set.

\begin{defi}[\cite{articlePQSym}]
The Hopf algebra $\defn{\PQSym}$ of\defn{parking quasi-symmetric functions}, whose bases are indexed by parking functions, has a basis $(\F_u)$ in which the product and coproduct are defined by
\begin{equation*}
\F_u \F_v = \sum_{w\,;\, \,w \in u \, \Cup \, v}  \F_w    
\end{equation*}
and
\begin{equation*}
 \Delta \F_u = \sum_{a,b \,;\, a\cdot b =u} \F_{\mathrm{Park}(a)} \otimes \F_{\mathrm{Park}(b)}   
\end{equation*}
 where $\cdot$ is the concatenation of words.
\end{defi}

The Hopf algebra $\PQSym$ is noncommutative and noncocommutative. It is graded by the length of parking functions and is connected. Its graded dual is then also a Hopf algebra, denoted by $\PQSymstar$. Let $\G_u := \F_u^*$ be the basis elements of the dual basis of $\PQSym$. The product and coproduct in $\PQSymstar$ satisfy

\begin{equation*}
\G_u  \G_v = \sum \limits_{\substack{w\in PF\, ;\, w=a\cdot b\\  \mathrm{Park}(a)=u,\, \mathrm{Park}(b)=v}} \G_w  
\end{equation*}
and
\begin{equation*}
\Delta \G_u = \sum_{a,b \,;\, u \in a \, \Cup \, b} \G_a \otimes \G_b.  
\end{equation*} 

For example, we have
\begin{align*}
\G_{11} \G_{21} &= \G_{1121} + \G_{1131} + \G_{1132} + \G_{1141} + \G_{1142} + \G_{1143} + \G_{2221} + \G_{2231} \\
&+ \G_{2241} + \G_{3321}. \\
 \Delta \G_{612441} &= 1 \otimes \G_{612441} + \G_{121} \otimes \G_{311} + \G_{12441} \otimes \G_{1} + \G_{612441} \otimes 1 .    
\end{align*}

Novelli and Thibon proved that $\PQSym^*$ is a bidendriform bialgebra (Theorem 3.6 of~\cite{novelli2005construction}), which implies, thanks to Foissy~\cite{foissy2007bidendriform}, that it is a self-dual Hopf algebra. In the sequel, we will focus on $\PQSymstar$ since, as opposed to $\PQSym$, it has a simple known realization in terms of noncommutative polynomials. Indeed, if $A$ is a totally ordered infinite alphabet, we set 
\begin{equation}
\label{Gsumwords}
 \G_u(A) := \sum_{w\in A^*, \,\mathrm{Park}(w)=u} w ,    
\end{equation}

\noindent where $A^*$ is the set of words on the alphabet $A$. Then the set of $\G_u(A)$ for $u\in PF$ generates a Hopf algebra isomorphic to $\PQSymstar$ so that $\PQSymstar$ can be seen as a subalgebra of $\mathbb{K}\langle A\rangle$.
For $A=\mathbb{N}\setminus \{0\}$, which will be always the case in the sequel, we have for example
$$\G_{113} = 113+114+115+\dots  + 224+225+226+\dots  + 335+336+337+\dots  = \sum \limits_{\substack{a\geq 1 \\ b\geq a+2}} aab .$$

\subsection{Some results of Hivert}
\label{resultsHivert}

We recall in this section some results of Hivert \cite{hivert1999combinatoire,hivertlocal} that motivated and are related to the results of this paper.

Let $X=\{x_1< x_2 < \cdots\}$ be a totally ordered infinite alphabet of commutative variables. The algebra $QSym$ of\defn{quasi-symmetric functions} can be defined as the vector space of bounded degree formal power series $F(x):=F(x_1, x_2,\dots)$ satisfying, for all $i_1<i_2<\dots<i_k$ and $j_1<j_2<\dots<j_k$,
$$\forall a_1,\dots,a_k \in \mathbb{N}, \,[x_{i_1}^{a_1} x_{i_2}^{a_2} \dots x_{i_k}^{a_k}] F(x) = [x_{j_1}^{a_1} x_{j_2}^{a_2} \dots x_{j_k}^{a_k}] F(x) ,$$

\noindent where for a monomial $s$, $[s]F(x)$ is the coefficient of $s$ in $F(x)$. We provide this vector space with the usual product of formal power series. With this definition, we deduce easily that the vectors
$$M_I := \sum_{j_1 < j_2 < \dots < j_r} x_{j_1}^{i_1} x_{j_2}^{i_2} \dots x_{j_r}^{i_r} ,$$

\noindent where $I=(i_1, i_2, \dots, i_r)$ is a integer composition, form a basis of $QSym$. This is the\defn{quasi-monomial basis}. For example $ M_{(1,2)}= x_1x_2^2+x_1x_3^2+\dots +x_2x_3^2+x_2x_4^2+\cdots$ .

When we use only a finite alphabet $X_n:=\{x_1< x_2 <\dots < x_n\}$, we have the algebra $QSym(X_n)$ of quasi-symmetric polynomials. The algebras $QSym$ and $QSym(X_n)$ are respectively subalgebras of $\mathbb{K}[X]$ and $\mathbb{K}[X_n]$. We can consider quasi-symmetric functions as a generalization of symmetric functions, a symmetric function being quasi-symmetric. Thus $Sym$ and $Sym(X_n)$ are respectively subalgebras of $QSym$ and $QSym(X_n)$.

Hivert defined \cite{hivert1999combinatoire} an action of the symmetric group $\mathfrak{S}_n$ on $\mathbb{K}[X_n]$, called\defn{the quasi-symmetrizing action}, such that the quasi-symmetric polynomials are the invariants. To define this action, we first code bijectively monomials as sequences of non negative integers with the identification $x_1^{k_1} x_2^{k_2} \dots x_n^{k_n} = [k_1, k_2, \dots, k_n]$. The quasi-symmetrizing action is defined by the action of the elementary transpositions $s_i:=(i,i+1)$ on monomials
\begin{equation}
s_i \cdot [k_1, \dots, k_i, k_{i+1}, \dots, k_n] = \left\{
    \begin{tabular}{ll}
        $[k_1, \dots, k_{i+1}, k_i, \dots, k_n] \,\text{if $k_i=0$ or $k_{i+1}=0$,}$\\
        
        $[k_1, \dots, k_i, k_{i+1}, \dots, k_n] \,\text{otherwise,}$
    \end{tabular}
\right.
\label{defaction}
\end{equation}
\noindent extended by linearity to an action on $\mathbb{K}[X_n]$.

Let us remark that this action is not compatible with the algebra structure, since for example $(s_1 \cdot x_1^2) (s_1 \cdot x_2) = x_1 x_2^2 \neq x_1^2 x_2= s_1 \cdot x_1^2 x_2$.
We remark that the quasi-symmetrizing action allows us only to move the zeros in the sequences $[k_1, \dots, k_n]$.

\begin{ex}
Let $n=4$. 

The orbit of $[4,1,0,0]$ is $\{[4,1,0,0], [4,0,1,0], [4,0,0,1], [0,4,1,0],
[0,4,0,1], [0,0,4,1]\}$ and
$$M_{(4,1)}=x_1^4x_2+x_1^4x_3+x_1^4x_4+x_2^4x_3+x_2^4x_4+x_3^4x_4 .$$
\end{ex}

The\defn{infinite symmetric group} $\frak{S}_{\infty}$ has as presentation $<s_1,s_2,\dots ,s_i,\dots |\,\mathcal{R}>$ with $s_i:=(i,i+1)$ the elementary transpositions and $\mathcal{R}$ the Moore-Coxeter relations.
In the case of an infinite alphabet $X$, by replacing monomials $[k_1, \dots, k_n]$ with infinite almost-zero sequences $[k_1,k_2, \dots]$ (all the $k_i$ are zeros for $i$ large enough) on which the elementary transpositions act like \eqref{defaction} we obtain an action of $\mathfrak{S}_{\infty}$ on $\mathbb{K}[ X]$, also called the\defn{quasi-symmetrizing action}. It is immediate that the quasi-monomial vectors are the sum of the elements in one of the orbits, and the elements of $QSym$ and $QSym(X_n)$ are the invariants.

Let $r\in \mathcal{N}$. What if we do not only allow the zeros to move in the sequences $[k_1, \dots, k_n]$ but all integers less than $r$? We also obtain a symmetric group action defined by 
\begin{equation}
s_i \underset{r}{\cdot} [k_1, \dots, k_i, k_{i+1}, \dots, k_n] = \left\{
    \begin{tabular}{ll}
        $[k_1, \dots, k_{i+1}, k_i, \dots, k_n] \,\text{if $k_i <r$ or $k_{i+1}<r$,}$\\
        
        $[k_1, \dots, k_i, k_{i+1}, \dots, k_n] \,\text{otherwise.}$
    \end{tabular}
\right.
\label{defraction}
\end{equation}

This action, and its extension by linearity to $\mathbb{K}[X_n]$, is called the\defn{$r$-quasi-symmetrizing action} . As done before, we also deduce an $r$-action in the case of an infinite alphabet $X$. The case $r=1$ is the quasi-symmetrizing action and the case $r=\infty$ is the usual action of permutation of the variables. By defining $QSym^r$ as the vector space of the invariants under the $r$-action, Hivert obtained the following inclusions of vector spaces
\begin{equation}
QSym = QSym^1 \supset QSym^2 \supset QSym^3 \supset ... \supset QSym^{\infty} = Sym  .
\label{interpolqsym}
\end{equation}

We can provide $QSym$ and $Sym$ with a coproduct which is compatible with the product (using for example the alphabet doubling trick) and thus $QSym$ and $Sym$ are both connected graded bialgebras, and thus Hopf algebras. 

It is easy to show that the orbits under the $r$-quasi-symmetrizing action are parameterized by\defn{$r$-compositions} $(I,\lambda)$, which are pairs of an integer composition $I$ with parts of size at least $r$ and an integer partition $\lambda$ with parts of size strictly less than $r$. We obtain a basis $M_{(I,\lambda)}^r$ of $QSym^r$ by summing the monomials in the orbit of $(I,\lambda)$. We have (which is Equation $(25)$ of \cite{hivertlocal}) 

$$ M_{(I,\lambda)}^r = \sum_{K \in I\, \shuffle \,\lambda^{\sigma}} M_K ,$$
where $\lambda^{\sigma}$ is the set of distinct permutations of $\lambda$.

This formula is very useful. It gives us a way to compute in $QSym^r$ by doing the computations in $QSym$, and proves that the subspaces $QSym^r$ have a basis whose vectors are disjoint sums of basis vectors of $QSym$. Hivert proved that the subspaces of \eqref{interpolqsym} are in fact Hopf subalgebras (\cite{hivertlocal}, Theorem $4.1$).
Using this formula and the fact that the $M_{(I,\lambda)}^r$ form a basis of $QSym^r$, he obtained (which is all together Proposition $3.8$, Theorem $3.10$, and Proposition $4.2$ of \cite{hivertlocal}):

\begin{thm}[\cite{hivertlocal}]
If $(I,\lambda)$ and $(J,\mu)$ are $r$-compositions, the product in $QSym^r$ is given by
$$M_{(I,\lambda)}^r M_{(J,\mu)}^r = \sum_{(K,\nu)} \sum \limits_{\substack{A \in I\,\shuffle \,\lambda^{\sigma}\\  B \in J\,\shuffle \,\mu^{\sigma}}} c_{(K,\nu),A,B} \,M_{(K,\nu)}^r  ,$$

\noindent where $c_{(K,\nu),A,B}$ is the coefficient of $(K,\nu)$ in the quasi-shuffle product of $A$ and $B$. The coproduct is given by
$$\Delta M_{(I,\lambda)}^r = \sum_{(K,K')=I\,, \, \lambda \in \nu \,\shuffle \,\nu'} M_{(K,\nu)}^r \otimes M_{(K',\nu')}^r .$$

The generating series of the dimensions of the homogeneous components of $QSym^r(X_n)$, where $X_n:=\{x_1<x_2<\cdots<x_n\}$, is given by
\begin{equation*}
\sum_{d,n} dim_d(QSym^r(X_n)) a^n t^d = \frac{1-t}{1-t-at^r} \prod_{i=0}^{r-1} \frac{1}{1-at^i} .    
\end{equation*}

\noindent In the case of $QSym^r$ (meaning $X$ infinite) the Hilbert series is
$$Hilb_t(QSym^r) = \sum_d dim_d(QSym^r)t^d = \frac{1}{1-t-t^r} \prod_{i=1}^{r-1} \frac{1}{1-t^i} .$$
\end{thm}

Let us now consider the noncommutative analogues. We replace the monomials $x_1^{k_1}\dots x_n^{k_n}=[k_1,\dots,k_n]$ by words on a totally ordered infinite alphabet of noncommutative variables $A=\{a_1<a_2<\dots \}$, or a finite one $A_n=\{a_1<a_2<\dots <a_n\}$.
Let $|w|_i$ be the number of occurrences of the letter $a_i$ in the word $w$ on $A_n$. We have the noncommutative analogue of the $r$-action
\begin{equation}
s_i \underset{r}{\cdot} w = \left\{
    \begin{tabular}{ll}
        $w'$ \,\text{if $|w|_i <r$ or $|w|_{i+1}<r$,}\\
        
        $w$ \,\text{otherwise,}
    \end{tabular}
\right.
\label{actionrwqsym}
\end{equation}

\noindent where $w'$ is the word $w$ where the $a_i$ letters are replaced by the letter $a_{i+1}$, and the $a_{i+1}$ letters are replaced by the letter $a_i$. 

\begin{ex}
Let $n=4$. We have
\begin{equation}
s_2  \underset{r}{\cdot} a_2a_1a_1a_3a_2a_4 =
\left\{
    \begin{tabular}{ll}
        \text{$a_3a_1a_1a_2a_3a_4$} \,\text{ if $r\geq 2$,}\\
        
     \text{$a_2a_1a_1a_3a_2a_4$}    \,\text{ if $r=1$.}
    \end{tabular}
\right.
\end{equation}
\end{ex}

By linearity, $\underset{r}{\cdot}$  defines an action of $\mathfrak{S}_n$ (or $\mathfrak{S}_{\infty}$ if $A$ is infinite) on $\mathbb{K}\langle A_n \rangle$ ($\mathbb{K}\langle A \rangle$ if $A$ is infinite).
Let $\textbf{WQSym}^r$ be the vector space of the invariants under the $r$-action. The case $r=1$ is the Hopf algebra $\textbf{WQSym}$ \cite{hivert1999combinatoire,novelli2005construction,bergeron2009hopf} and the case $r=\infty$ is the Hopf algebra $WSym$ \cite{10.1215/S0012-7094-36-00253-3,bergeron2009hopf}. 

We recall that packing is an algorithm mapping words on $A$ to packed words, a packed word being a word $m$ on positive integers such that if $i>1$ is in $m$, then $i-1$ also.
The packing of a word $w$ on $A$, written $\mathrm{Pack}(w)$, is the image by the morphism $w_i \mapsto i$ of the word $w$, where $w_1<\dots <w_k$ are the letters appearing in $w$. For example, we have $\mathrm{Pack}(a_7a_3a_9a_5a_5a_2a_5) = 4253313$.

A basis of $\textbf{WQSym}$ is given by the $\textbf{WQM}_u$, where $u$ is a packed word on $A$, defined by
$$\textbf{WQM}_u := \sum_{w\,;\,\mathrm{Pack}(w)=u} w .$$

\noindent It is immediate that we have the following inclusions of vector spaces
\begin{equation*}
\textbf{WQSym} = \textbf{WQSym}^1 \supset \textbf{WQSym}^2 \supset \textbf{WQSym}^3 \supset ... \supset \textbf{WQSym}^{\infty} = \textbf{WSym}  .
\label{interpolwqsym}
\end{equation*}

\noindent As in the commutative case, we have a very useful formula 
\begin{equation*}
\textbf{WQM}_{(I,\lambda)}^{\, r} = \sum_{K \in I\, \shuffle \,\lambda^{\sigma}} \textbf{WQM}_K ,   
\end{equation*}
where $I$ is a set composition with parts of size at least $r$, $\lambda$ a set partition with parts of size strictly less than $r$ and $\lambda^{\sigma}$ is the set of all set compositions form by using the parts of $\lambda$. The shuffle product is a shuffle of set compositions, which means that $I\, \shuffle \,\lambda '$, for $\lambda ' \in \lambda^{\sigma}$, is the set of all set compositions where the parts are the parts of $I$ and $\lambda '$ with the same order on the parts coming from $I$, and the ones coming from $\lambda '$ . We can then compute in $\textbf{WQSym}^r$ and Hivert proved (\cite{hivertlocal}, Theorem $7.1$) that the above inclusions are inclusions of Hopf algebras.

\section{The free quasi-symmetrizing action}
\label{actionperm}

In this section we define a new symmetric group action on formal power series in non-commuting variables such that the elements of the Hopf algebra $\textbf{FQSym}^*$ are the invariants. 
We begin by defining a bijection between words on positive integers and certain almost zero sequences.

\begin{defi}
Let $E$ be the set of all infinite almost zero sequences $a=(a_1,a_2,\dots)$ of non-negative integers such that the subword $a_{j_1}\dots a_{j_k}$ of the non-zero $a_i$, called the\defn{associated permutation}, is a permutation of $[k]$. 
\end{defi}

\begin{ex}
We have $(0,6,4,0,0,2,3,1,7,0,5,0,\dots) \in E$ because the associated permutation is $6423175 \in \mathfrak{S}_7$. 
\end{ex}

\begin{defi}
\label{defbij1}
Let $w=w_1\dots w_n$ be a word on positive integers and $\tau:=\mathrm{Std}(w)^{-1}$. We now define an element $f(w)$ in $E$ having $\tau$ as its associated permutation. This requires additionally that we specify the positions of the zeros. For this, we scan $\tau$ from left to right and at each step $i$, for $i\in [n-1]$, we have two possibilities:
\begin{itemize}
    \item[$(1)$] if $\tau_i < \tau_{i+1}$, then we put $w_{\tau_{i+1}}-w_{\tau_i}$ zeros between $\tau_j$ and $\tau_{i+1}$ in $f(w)$,
    \item[$(2)$] otherwise $\tau_i > \tau_{i+1}$, then we put $(w_{\tau_{i+1}}-w_{\tau_i})-1$ zeros between $\tau_i$ and $\tau_{i+1}$ in $f(w)$.
\end{itemize}
We also add $w_{\tau_1}-1$ zeros before $\tau_1$. Then we obtain $f(w)$ by putting an infinite number of zeros at the right of $\tau_n$.
\end{defi}

\begin{ex}
\label{exF}
Let $w=972554=w_1\dots w_6$. We have $\mathrm{Std}(w)^{-1}=364521=\tau_1\dots \tau_6$. So $f(w)$ is equal to $(3,6,4,5,2,1)$ with extra zeros added. Since $3<6$, we are in case $(1)$, and thus add $w_6-w_3= 4-2=2$ zeros between $3$ and $6$ and obtain $(3,0,0,6,4,5,2,1)$. Since $6>4$, we are in case $(2)$, and thus add $(w_4-w_6)-1=(5-4)-1=0$ zero between $6$ and $4$ and thus obtain $(3,0,0,6,4,5,2,1)$. Since $4<5$, we are in case $(1)$, and thus add $w_5-w_4=5-5=0$ zero between $4$ and $5$ and thus obtain $(3,0,0,6,4,5,2,1)$. 
By continuing the algorithm, we will finally obtain
$$f(972554)=(0,3,0,0,6,4,5,0,2,0,1,0,\dots) .$$
\end{ex}

\begin{defi}
\label{defrevbij1}
Let $a=(a_1,a_2,\dots)\in E$ with $a_{j_1}\dots a_{j_k}$ its associated permutation. From $a$, we construct a word of length $k$ on positive integers $g(a):=w_1 \dots w_k$. We start by setting $w_{a_{j_1}}:=j_1$. Then at step $i+1$ with $i\in [k-1]$, we already know the letters $w_{a_{j_1}}$, $w_{a_{j_2}}$, $\dots$, $w_{a_{j_i}}$ of $g(a)$ and we set
\begin{equation*}
w_{a_{j_{i+1}}}:= \left\{\begin{tabular}{ll}
    $w_{a_{j_i}}+(j_{i+1}-j_i)-1$ &\text{if }$\,a_{j_{i}}<a_{j_{i+1}}$,   \\
    $w_{a_{j_i}}+(j_{i+1}-j_i)$ &\text{otherwise.} 
\end{tabular}
\right.
\end{equation*}
\end{defi}

\begin{ex}
\label{exG}
Let $a=(0,3,0,0,6,4,5,0,2,0,1,0,\dots)=(a_1,a_2,\dots)\in E$. The associated permutation is $364521=a_2a_5a_6a_7a_9a_{11}$, thus $g(a)$ is a word of length $6$. Let $w_i$ be the $i^{th}$ letter of $g(a)$. We have $w_{a_2}=w_3=2$ and then $g(a)=\_\, \_\, 2\, \_\, \_\, \_$. We have $w_{a_5}=w_6=2+(5-2)-1=4$ since $a_2=3<6=a_5$, and then $g(a)=\_\, \_\, 2\, \_\, \_\, 4$. We have $w_{a_6}=w_4=4+(6-5)=5$ since $6>4$, and then $g(a)=\_\, \_\, 2\, 5\, \_\, 4$. We have $w_{a_7}=w_5=5+(7-6)-1=5$ since $4<5$, and then $g(a)=\_\, \_\, 2\, 5\, 5\, 4$. We have $w_{a_9}=w_2=5+(9-7)=7$ since $5>2$, and then $g(a)=\_\, 7\, 2\, 5\, 5\, 4$. We have $w_{a_{11}}=w_1=7+(11-9)=9$ since $2>1$, and then $g(a)=9\, 7\, 2\, 5\, 5\, 4$.
\end{ex}

 We remark with Examples \ref{exF} and \ref{exG} that $f(972554)=(0,3,0,0,6,4,5,0,2,0,1,0,\dots)$ and $g((0,3,0,0,6,4,5,0,2,0,1,0,\dots))=9 7 2 5 5 4$. It is easy to prove the following:

\begin{prop}
The maps $f$ and $g$ are inverses of each other.
\end{prop}

Thus, we encoded bijectively all words on $\mathbb{N}\setminus \{0\}$ by elements of $E$. From now on, we will forget about the bijection $f$ and identify a word $w$ with $a\in E$ when $f(w)=a$.

\begin{prop}
\label{seqstandardization}
Let $a\in E$ and $\sigma$ be its associated permutation. Then the standardization of $a$ is $\sigma^{-1}$. 
The sequences in $E$ that have the same standardization are those with the same associated permutation.
\end{prop}

\begin{proof}
It is by definition of the  bijection $f$. The second part is then a direct consequence of the first one.
\end{proof}

To define a group action on words on $\mathbb{N}\setminus \{0\}$ we define it on $E$.

\begin{prop}
\label{acfqsym}
We define an action $\cdot$ of $\frak{S}_{\infty}$ on $E$ (and thus on words on $\mathbb{N}\setminus \{0\}$) by giving the action of the generators $s_i$. The elementary transposition $s_i$ acts on $a\in E$ by swapping the $i^{th}$ and $(i+1)^{st}$ coordinates if one of these coordinates is zero, and doing nothing otherwise.
\end{prop}

\begin{proof}
We could verify that the Moore-Coxeter relations are satisfied, but we can prove that it defines an action directly by making a link with the action by permutation of the variables. We have a bijection between $E$ and the set of pairs with first coordinate an almost zero sequence containing only $1$ and $0$'s and second coordinate a permutation of $\frak{S}_k$ where $k$ is the number of $1$'s in the almost zero sequence. This bijection sends $a\in E$ to the pair $(A,B)$ where $A$ is $a$ with each positive element replaced by $1$ and $B$ is the associated permutation of $a$. Then $\cdot$ is just the action on the pairs that acts by permuting the entries for the first coordinate and does not change the second coordinate, which proves that it is an action.
\end{proof}

\begin{defi}
We call the above action, and its extension by linearity to sums of words on $\mathbb{N}\setminus \{0\}$, the\defn{free quasi-symmetrizing action}.
\end{defi}

\begin{ex}
We have $s_2 \cdot (1,0,2,0,\dots) = (1,2,0,\dots)$, meaning $s_2 \cdot 12 = 11$.
We have $s_2 \,\cdot \,(0,3,0,0,6,4,5,0,2,0,1,0,\dots) = (0,0,3,0,6,4,5,0,2,0,1,0,\dots)$, meaning $s_2 \cdot 972554 = 973554$.
\end{ex}

\begin{thm}
\label{thmFQSym}
The orbit of $w$, denoted by $\mathrm{Orb}(w)$, is the set of words whose standardization is $\mathrm{Std}(w)$. If $\displaystyle \sigma \in \bigcup_{n\geq 0} \frak{S}_n$, we have
$$ \G_{\sigma} =\sum_{w\,, \,w\in \mathrm{Orb}(\sigma)} w .$$
An (infinite) sum of words on positive integers is in $\textbf{FQSym}^*$ if and only if it is invariant under the free quasi-symmetrizing action.
\end{thm}

\begin{proof}
The first part and the formula for $\G_{\sigma}$ are a consequence of Proposition~\ref{seqstandardization} and the polynomial realization given by Formula~\eqref{polyrealfq}.
It implies that we have a basis of $\textbf{FQSym}^*$ given by the sum on the orbits and since each orbit exactly contains one permutation, any invariant element is a linear combination of~$\G_{\sigma}$.
\end{proof}

\begin{remark}
We can define, for $r\in \mathcal{N}$, a more general symmetric action on $E$ such that $s_i$ swaps the $i^{th}$ and $(i+1)^{st}$ coordinates if one of these coordinates is strictly less than $r$, and does nothing otherwise. The case $r=1$ thus corresponds to the free quasi-symmetrizing action. We can prove that the subspaces of the invariants for these generalized actions are subalgebras. But they are not subcoalgebras of $\textbf{FQSym}^*$. Thus they do not define Hopf subalgebras.
\end{remark}

\section{The parking quasi-symmetrizing action}
\label{sectionparkingquasi}

In this section, we define an action of the infinite symmetric group on the set of words on positive integers such that the elements of $\PQSymstar$ are the invariant word sums. We first encode bijectively the words on positive integers as bi-words, then define this action on these bi-words. Recall that a set composition of $[n]:=\{1,2,...,n\}$ is an ordered set partition of $[n]$.

\begin{defi}
\label{definitionBW}
Let $\mathrm{BW}$ be the set of all infinite two row arrays $m=\begin{tabular}{|ccc|}
    $L_1$ & $L_2$ & $\cdots$  \\
    $M_1$ & $M_2$ & $\cdots$  
\end{tabular}$ where $L_i$ is a set and $M_i$ is a word on positive integers such that

\begin{itemize}
    \item[$(1)$] the sequence of all non-empty $L_i$ is a set composition of a set $[n]$,
    \item[$(2)$] $M_i$ is a prime parking function of length the number of elements of $L_i$.
\end{itemize}

\noindent The\defn{length of $m$} is the integer $n$ of the first point. Let $\mathrm{BW}_n$ be the set of $m\in \mathrm{BW}$ of length $n$.
\end{defi}

In the sequel, we shall represent an element of $\mathrm{BW}$ as a bi-word: any set $L_i$ will be represented as the concatenation of its elements in increasing order.

\begin{ex}
The two-row array $\begin{tabular}{|cccccccc|}
    $\{3,4,6\}$ & $\{2\}$ & $\emptyset$ & $\{1,7\}$ & $\emptyset$ & $\{5\}$ & $\emptyset$ & $\cdots$ \\
    $121$ & $1$ & $\epsilon$ & $11$ & $\epsilon$ & $1$ & $\epsilon$ & $\cdots$ 
\end{tabular}$ is in $\mathrm{BW}_7$ since the first row is the set composition $(\{3,4,6\},\{2\},\{1,7\},\{5\})$ of $[7]$ and $1$, $11$, and $121$ are prime parking functions. We shall represent it as the bi-word \

$\begin{tabular}{|cccccccc|}
    $346$ & $2$ & $\epsilon$ & $17$ & $\epsilon$ & $5$ & $\epsilon$ & $\cdots$ \\
    $121$ & $1$ & $\epsilon$ & $11$ & $\epsilon$ & $1$ & $\epsilon$ & $\cdots$ 
\end{tabular}\,.$
\end{ex}

In the sequel, an\defn{empty column} is a column whose words are the empty words $\epsilon$. If $u=u_1\dots u_n\in PF_n$ has $k$ prime blocks denoted by $U_i$ 
for all $i\in [k]$, let $P_i$
be the word such that its $j^{th}$ letter, denoted $(P_i)_j$, is the position in $u$ of the $j^{th}$ letter of $U_i$ (see Example \ref{expsi}). 

\begin{defi}
Let $\psi:u\in PF \longmapsto \begin{tabular}{|ccccc|}
    $P_1$ & $\cdots$ & $P_k$ & $\epsilon$ & $\cdots$ \\
    $\mathrm{Park}(U_1)$ & $\cdots$  & $\mathrm{Park}(U_{k})$ & $\epsilon$ & $\cdots$ 
\end{tabular} \in \mathrm{BW}$.
\end{defi}

\begin{ex}
\label{expsi}
Let $u=5412715\in PF$. We have $u^{\uparrow} = 112|4|55|7$ and the four prime blocks of $u$ are $U_1=121$, $U_2=4$, $U_3=55$, and $U_4=7$. These are respectively at positions $P_1=346$, $P_2=2$, $P_3=17$, and $P_4=5$. Then
$\psi(u)~=~\begin{tabular}{|cccccc|}
    $346$ & $2$ & $17$ & $5$ & $\epsilon$ & $\cdots$ \\
    $121$ & $1$ & $11$ & $1$ & $\epsilon$ & $\cdots$ 
\end{tabular}\,.$
\end{ex}

If $w=w_1\dots w_n$ is a word on positive integers such that the prime blocks of $\mathrm{Park}(w)=u_1\dots u_n$ are $U_i$
for all $i\in [k]$, let $d_i:=w_{(P_i)_1} - u_{(P_i)_1} \geq 0$
for all $i\in~[k]$. A consequence of the parkization algorithm is the following.

\begin{lem}
\label{lemmaparkization}
With the notations as before, for all $i\in [k]$, for all $j$, $d_i=w_{(P_i)_j} - u_{(P_i)_j}$ and the $d_i$ are nondecreasing.
\end{lem}

Note that $w\in PF$ if and only if $d_i=0$ for all $i\in~[k]$. By convention, let $|\epsilon|:=1$. If $L_p\neq \epsilon$, let $|L_p|$ be the length of $L_p$. Thanks to Lemma \ref{lemmaparkization} we have the following.

\begin{prop}
\label{bijectionphi}
We have a bijection $\phi:$ finite words on $\mathbb{N}\setminus \{0\} \longrightarrow \mathrm{BW}$ defined by $$\phi(w)= \begin{tabular}{|ccccccc|}
    $\epsilon$ & $\cdots$  & $P_1$ & $\epsilon$ & $\cdots$ & $P_2$ & $\cdots$  \\
    $\epsilon$ & $\cdots$ & $\mathrm{Park}(U_1)$ & $\epsilon$ & $\cdots$ & $\mathrm{Park}(U_2)$ & $\cdots$ 
\end{tabular} \,,$$

\noindent where $\phi(w)$ is the bi-word $\psi(\mathrm{Park}(w))$ with $d_i$ empty columns to the left of column $\begin{tabular}{c}
    $P_i$ \\
    $\mathrm{Park}(U_i)$ 
\end{tabular} \,$ for all $i$. If $m=\begin{tabular}{|ccc|}
    $L_1$ & $L_2$ & $\cdots$  \\
    $M_1$ & $M_2$ & $\cdots$ 
\end{tabular} \in \mathrm{BW}_n$ with $k$ non-empty columns of indices $j_1<\dots<j_k$, the word
$w:=\phi^{-1}(m)$ is of length $n$ and defined by putting at positions given by the letters of $L_{j_i}$, for all $i\in [k]$, the letters of $M_{j_i}$ all incremented by $\displaystyle \sum_{p=1}^{j_{i}-1} |L_p|$.
\end{prop}

Note that since the $d_i$ are nondecreasing, there is a unique way to add empty columns to satisfy the statement: put $d_1$ empty columns to the left of the column containing $P_1$ and $d_i-d_{i-1}$ between the columns containing $P_{i-1}$ and $P_i$ for all $i\geq 2$.

\begin{ex}
\label{exbijectionbiwordparking}
Let $w=6412916$. We have $\mathrm{Park}(w)=5412715$. Thanks to Example \ref{expsi}, we have 
$\psi(\mathrm{Park}(w))=
\begin{tabular}{|cccccc|}
    $346$ & $2$ & $17$ & $5$ & $\epsilon$ & $\cdots$ \\
    $121$ & $1$ & $11$ & $1$ & $\epsilon$ & $\cdots$ 
\end{tabular}\,.$ We have $d_1=0$, $d_2=0$, $d_3=1$, and $d_4=2$. Then
$\phi(w)=
\begin{tabular}{|cccccccc|}
    $346$ & $2$ & $\epsilon$ & $17$ & $\epsilon$ & $5$ & $\epsilon$ & $\cdots$ \\
    $121$ & $1$ & $\epsilon$ & $11$ & $\epsilon$ & $1$ & $\epsilon$ & $\cdots$ 
\end{tabular} \,.$

 Conversely, let $m=\begin{tabular}{|cccccccc|}
    $346$ & $2$ & $\epsilon$ & $17$ & $\epsilon$ & $5$ & $\epsilon$ & $\cdots$ \\
    $121$ & $1$ & $\epsilon$ & $11$ & $\epsilon$ & $1$ & $\epsilon$ & $\cdots$ 
\end{tabular}\in~\mathrm{BW}$ and $w:=\phi^{-1}(m)$. The first row tells us that the length of $w$ is $7$. The first column gives us $w=\_\, \_ 1 2 \,\_\, 1 \,\_$ . The second column gives us $w=\_\, 4 1 2 \,\_\, 1 \,\_$ since at position $2$ in $w$ we put the letter $1$, which is the letter of the second row of the second column, incremented by $|346|=3$. The fourth column gives us $w=6 4 1 2 \,\_\, 1 6$ since at positions $1$ and $7$ we put the letter $1$ incremented by $|346|+|2|+|\epsilon|=3+1+1=5$. Finally, the sixth column gives us $w=6412916$ since at position $5$ we put the letter $1$ incremented by $|346|+|2|+|\epsilon|+|17|+|\epsilon|=8$. 
\end{ex}

\begin{remark}
The bijection $\phi$ preserves the size: $\phi$ is a bijection between words of length $n$ on $\mathbb{N}\setminus \{0\}$ and $\mathrm{BW}_n$.
\end{remark}

From now on, we will often forget about the bijection $\phi$ and identify $w$ with $m$ when $\phi(w)=m$ (see Example \ref{exbijectionbiwordparking}).

\begin{prop}
\label{biwordparkization}
Let $u\in PF$ be a parking function. The bi-words in $\mathrm{BW}$ whose parkization is $u$ are those with the same non-empty columns and in the same order as $\phi(u)$.
\end{prop}

\begin{proof}
It is a consequence of the bijection $\phi$. The empty columns encode the difference between a word and its parkized word. The parkization algorithm consists of deleting the empty columns appearing to the left of a non-empty column, until there are no such columns anymore.
\end{proof}

\begin{ex}
The two words $2235559$ and $2236669$, which are respectively the two bi-words $\begin{tabular}{|ccccccc|}
    $\epsilon$ & $123$ & $456$ & $\epsilon$ & $7$ & $\epsilon$ & $\cdots$ \\
    $\epsilon$ & $112$ & $111$ & $\epsilon$ & $1$ & $\epsilon$ & $\cdots$ 
\end{tabular}$ and  $\begin{tabular}{|ccccccc|}
    $\epsilon$ & $123$ & $\epsilon$ & $456$ & $7$ & $\epsilon$ & $\cdots$ \\
    $\epsilon$ & $112$ & $\epsilon$ & $111$ & $1$ & $\epsilon$ & $\cdots$ 
\end{tabular}\,,$ have the same parkization $$1124447 = \phi^{-1} \left(\,\begin{tabular}{|ccccc|}
     $123$ & $456$ & $7$ & $\epsilon$ & $\cdots$ \\
     $112$ & $111$ & $1$ & $\epsilon$ & $\cdots$ 
\end{tabular}\,\right) \,.$$
\end{ex}

Proposition~\ref{biwordparkization} shows that a way to move between all words having the same parkization is to exchange in a bi-word empty columns with non-empty ones.

\begin{prop}
\label{defactionparkingtranspo}
We define an action $\cdot$ of $\frak{S}_{\infty}$ on $\mathrm{BW}$ (and thus on words on $\mathbb{N}\setminus \{0\}$) by giving the action of the generators $s_i$. The elementary transposition $s_i$ acts on $m\in \mathrm{BW}$ by swapping the $i^{th}$ and $(i+1)^{st}$ columns if one of these columns is empty, and doing nothing otherwise.
\end{prop}

\begin{proof}
    A very similar proof to the one in Proposition~\ref{acfqsym} can be made.
\end{proof}

\begin{defi}
The above action, and its extension by linearity to sums of bi-words, is called the\defn{parking quasi-symmetrizing action}.
\end{defi}

\begin{ex}
We have

\begin{equation*}
 s_i \,\cdot\, \begin{tabular}{|ccccccc|}
    $\epsilon$ & $123$ & $456$ & $\epsilon$ & $7$ & $\epsilon$ & $\cdots$\\
    $\epsilon$ & $112$ & $111$ & $\epsilon$ & $1$ & $\epsilon$ & $\cdots$
\end{tabular} 
=  
\left\{
\begin{tabular}{lllll}
     \begin{tabular}{|ccccccc|}
    $123$ & $\epsilon$ & $456$ & $\epsilon$ & $7$ & $\epsilon$ & $\cdots$\\
    $112$ & $\epsilon$ & $111$ & $\epsilon$ & $1$ & $\epsilon$ & $\cdots$
\end{tabular}   $\,\text{if }\,i=1$,\\
\\
     \begin{tabular}{|ccccccc|}
    $\epsilon$ & $123$ & $\epsilon$ & $456$ & $7$ & $\epsilon$ & $\cdots$\\
    $\epsilon$ & $112$ & $\epsilon$ & $111$ & $1$ & $\epsilon$ & $\cdots$
\end{tabular} $\,\text{if }\,i=3$, \\
\\
\begin{tabular}{|ccccccc|}
    $\epsilon$ & $123$ & $456$ & $7$ & $\epsilon$ & $\epsilon$ & $\cdots$\\
    $\epsilon$ & $112$ & $111$ & $1$ & $\epsilon$ & $\epsilon$ & $\cdots$
\end{tabular}  $ \,\text{if }\,i=4$,\\
\\
    \begin{tabular}{|ccccccc|}
    $\epsilon$ & $123$ & $456$ & $\epsilon$ & $\epsilon$ & $7$ & $\cdots$\\
    $\epsilon$ & $112$ & $111$ & $\epsilon$ & $\epsilon$ & $1$ & $\cdots$
\end{tabular}   $\,\text{if }\,i=5$,\\
\\
    \begin{tabular}{|ccccccc|}
    $\epsilon$ & $123$ & $456$ & $\epsilon$ & $7$ & $\epsilon$ & $\cdots$\\
    $\epsilon$ & $112$ & $111$ & $\epsilon$ & $1$ & $\epsilon$ & $\cdots$
\end{tabular}  \,\text{otherwise}.
\end{tabular}
\right .   
\end{equation*}

\begin{equation*}
\text{Meaning}\,\,\, s_i \,\cdot\, 2235559 = s_i \,\cdot\, \begin{tabular}{|ccccccc|}
    $\epsilon$ & $123$ & $456$ & $\epsilon$ & $7$ & $\epsilon$ & $\cdots$\\
    $\epsilon$ & $112$ & $111$ & $\epsilon$ & $1$ & $\epsilon$ & $\cdots$
\end{tabular} 
=  
\left\{
\begin{tabular}{lllll}
$1125559  \,\,\text{if }\,i=1$ ,\\
$2236669  \,\,\text{if }\,i=3$ , \\
$2235558   \,\,\text{if }\,i=4$ ,\\
$223555.10   \,\,\text{if }\,i=5$ ,\\
$2235559  \,\,\text{otherwise} $.
\end{tabular}
\right .   
\end{equation*}
\end{ex}

\begin{thm}
\label{thmpolynomialrealization}
The orbit of $w$, denoted by $\mathrm{Orb}(w)$, is the set of words whose parkization is $\mathrm{Park}(w)$. If $u\in PF$, we have
\begin{equation*}
 \G_u =\sum_{w\,, \,w\in \mathrm{Orb}(u)} w .  
\end{equation*}
An (infinite) sum of words on positive integers is in $\textbf{PQSym}^*$ if and only if it is invariant under the parking quasi-symmetrizing action.
\end{thm}

\begin{proof}
The first part and the formula for $\G_u$ are a consequence of Proposition~\ref{biwordparkization} and the polynomial realization given by Formula~\eqref{Gsumwords}.
It implies that we have a basis of $\PQSymstar$ given by the sum on the orbits and since each orbit exactly contains one parking function, any invariant element is a linear combination of~$\G_u$.
\end{proof}

\section{The Hopf algebras ${\PQSymstar}^r$}
\label{sectionactionparkingrR}

In this section, we generalize the parking quasi-symmetrizing action with a parameter $r\in\mathcal{N}$, where the case $r=1$ gives back the parking quasi-symmetrizing action. In the sequel, $r\in\mathcal{N}$.

\begin{prop}
We define an action $\underset{r}{\cdot}$ of $\frak{S}_{\infty}$ on $\mathrm{BW}$ (and thus on words on $\mathbb{N}\setminus \{0\}$) by making the elementary transpositions $s_i$ act on $m\in \mathrm{BW}$ by swapping the $i^{th}$ and $(i+1)^{st}$ columns if one of these columns has words of length strictly smaller than $r$, and doing nothing otherwise.
\end{prop}

\begin{defi}
The above action, and its extension by linearity to sum of bi-words, is called the\defn{$r$-parking quasi-symmetrizing action} (abbreviated as\defn{$r$-action} in the sequel).
\end{defi}

\begin{ex}
We have
$s_2  \,\underset{2}{\cdot}\, \begin{tabular}{|ccccc|}
     $123$ & $456$ & $7$ & $\epsilon$ & $\cdots$ \\
     $112$ & $111$ & $1$ & $\epsilon$ & $\cdots$ 
\end{tabular} = \begin{tabular}{|ccccc|}
     $123$ & $7$ & $456$ & $\epsilon$ & $\cdots$ \\
     $112$ & $1$ & $111$ & $\epsilon$ & $\cdots$ 
\end{tabular}$ , meaning 
$s_2  \,\underset{2}{\cdot}\, 1124447 = 1125554$.
\end{ex}

\begin{remark}
\label{remarkinclusionofspaces}
 The case $r=1$ is the parking quasi-symmetrizing action defined above and for all $r\geq 2$, if a sum of elements of $\mathrm{BW}$ (or words on positive integers) is invariant under the $r$-action, then it is invariant under the $(r-1)$-action. Indeed, the $r$-action allows us to move between bi-words of $\mathrm{BW}$ having the same non-empty columns whose non-empty columns with words of length at least $r$ appear in the same order.
\end{remark}

\begin{defi}
We define ${\PQSymstar}^r$ as the linear subspace of $\PQSymstar$ of the invariants under the $r$-action. Their elements are the\defn{$r$-parking quasi-symmetric functions}.
\end{defi}

Thanks to Remark~\ref{remarkinclusionofspaces}, we obtain

\begin{prop}
We have an infinite chain of nested graded subspaces
\begin{equation*}
  \PQSymstar = {\PQSymstar}^1  \supseteq {\PQSymstar}^2 \supseteq \dots  \supseteq {\PQSymstar}^r \supseteq \dots  \supseteq {\PQSymstar}^{\infty} .
\end{equation*}
\end{prop}

We now study these subspaces. To describe a basis (Proposition~\ref{propbasepqsymr}), we introduce the following definition.

\begin{defi}
An\defn{$r$-bi-word of length $n$} is a pair $(I,\lambda)$ of two finite two row arrays of the form 
\begin{tabular}{|cccc|}
    $L_1$ & $L_2$ & $\cdots$  & $L_k$ \\
    $M_1$ & $M_2$ & $\cdots$  & $M_k$
\end{tabular} where $M_i$ satisfies point $(2)$ of Definition~\ref{definitionBW} and such that
\begin{itemize}
    \item the $L_i$ appearing in $I$ are sets with at least $r$ elements and those appearing in $\lambda$ are non-empty sets with strictly fewer than $r$ elements,
    \item the sets of the first row of $\lambda$ are sorted such that the smallest letters of these sets form an increasing sequence,
    \item the sequence of all sets of the first row of $I$ and the first row of $\lambda$ form a set composition of $[n]$.
\end{itemize}
\end{defi}

As previously done, any set $L_i$ will be represented as the concatenation of its elements in increasing order.
We will identify the $r$-bi-word $(I,\lambda)$ with the bi-word in $\mathrm{BW}$ which is the concatenation of $I$ and $\lambda$ to which we add an infinite number of empty columns to the right (to obtain an element of $\mathrm{BW}$). In the sequel, $(I,\lambda)$ is always an $r$-bi-word.

\begin{ex}
The pair $(I,\lambda)$ with $I=$ \begin{tabular}{|cc|}
    $36$ & $257$   \\
    $11$ & $121$
\end{tabular}
and $\lambda =$ \begin{tabular}{|cc|}
    $1$ & $4$  \\
    $1$ & $1$ 
\end{tabular} is a $2$-bi-word of length $7$. We have the identification $\left(\,\,\begin{tabular}{|cc|}
    $36$ & $257$   \\
    $11$ & $121$
\end{tabular} \,\,,\,\,\begin{tabular}{|cc|}
    $1$ & $4$  \\
    $1$ & $1$ 
\end{tabular}\,\, \right) = \begin{tabular}{|cccccc|}
    $36$ & $257$ & $1$ & $4$ & $\epsilon$ & $\cdots$   \\
    $11$ & $121$ & $1$ & $1$ & $\epsilon$ & $\cdots$ 
\end{tabular} \,.$
\end{ex}

Denote by $\mathrm{Orb}^r(m)$ the orbit of the bi-word $m$ under the $r$-action.

\begin{prop}
\label{propbasepqsymr}
The bases of ${\PQSymstar}^r$ are indexed by $r$-bi-words, the homogeneous components of degree $n$ having their bases indexed by $r$-bi-words of length $n$. A basis is given by the vectors 
$$\displaystyle \G_{(I,\lambda)}^{(r)} := \sum_{w\in \mathrm{Orb}^r((I,\lambda))} w .$$
\end{prop}

\begin{prop}
\label{propdimensionpqsymr}
The Hilbert series of ${\PQSymstar}^r$ is $\displaystyle \sum_{n\geq 0}  dim({\PQSymstar_n}^r) t^n =1+ \sum_{n\geq 1} \left(\sum_{k=1}^n A_{n,k}^{\,r}\right) t^n$, where $$\displaystyle A_{n,k}^{\,r} := \sum \limits_{\substack{\{Q_1,Q_2,\dots ,Q_k\}\,\text{partition of} \\ [n] \, \text{in}\,k\,\text{parts}}} \#\{i,\,|Q_i|\geq r\} \,!\,\,\,\prod_{i=1}^k (|Q_i|-1)^{|Q_i|-1} .$$
\end{prop}

\begin{proof}
Recall from Proposition~\ref{cardinalparkingfunctions} that the number of prime parking functions of length $|Q_i|$ is $(|Q_i|-1)^{|Q_i|-1}$. It is easy to prove that the number of $r$-bi-words of length $n$ with $k$ non-empty columns is given by $A_{n,k}^r$, which proves the statement by summing over~$k$.
\end{proof}

\begin{prop}
\label{propreduce}
We have 
\begin{equation}
\label{formR}
\displaystyle \G_{(I,\lambda)}^{(r)} = \sum_{K \in I\,\shuffle \, \lambda^{\sigma}} \G_K ,    
\end{equation}

\noindent where $\lambda^{\sigma}$ is the set of all bi-words whose columns are permutations of those of $\lambda$ and the shuffle is a shuffle of bi-words (the letters are the columns).
\end{prop}

\begin{ex}
\noindent For example, if $r=2$ we have
\begin{align*}
\G_{\left(\,\,\begin{tabular}{|c|}
    $23$   \\
    $11$
\end{tabular} \,\,,\,\,\begin{tabular}{|cc|}
    $1$ & $4$  \\
    $1$ & $1$ 
\end{tabular}\,\, \right)}^{\,(2)}
&= \G_{\,\,\begin{tabular}{|cccc|}
    $23$ & $1$ & $4$ & $\cdots$  \\
    $11$ & $1$ & $1$ & $\cdots$ 
\end{tabular}} \,+\,
\G_{\,\,\begin{tabular}{|cccc|}
    $1$ & $23$ & $4$ & $\cdots$  \\
    $1$ & $11$ & $1$ & $\cdots$ 
\end{tabular}} \\
&+ \G_{\,\,\begin{tabular}{|cccc|}
    $1$ & $4$ & $23$ & $\cdots$  \\
    $1$ & $1$ & $11$ & $\cdots$ 
\end{tabular}} \,+\,
 \G_{\,\,\begin{tabular}{|cccc|}
    $23$ & $4$ & $1$ & $\cdots$  \\
    $11$ & $1$ & $1$ & $\cdots$ 
\end{tabular}} \\
&+ \G_{\,\,\begin{tabular}{|cccc|}
    $4$ & $23$ & $1$ & $\cdots$  \\
    $1$ & $11$ & $1$ & $\cdots$ 
\end{tabular}} \,+\,
\G_{\,\,\begin{tabular}{|cccc|}
    $4$ & $1$ & $23$ & $\cdots$  \\
    $1$ & $1$ & $11$ & $\cdots$ 
\end{tabular}} \,,
\end{align*}

\noindent that translates as $\G_{3114}^{\,(2)} =  \G_{3114} + \G_{1224} +
\G_{1332} +\G_{4113} +\G_{4221} + \G_{2331}$.
\end{ex}

We now look at the algebra and coalgebra structures of these subspaces.

\begin{prop}
The subspace ${\PQSymstar}^r$ is a subcoalgebra of $\PQSymstar$. We have
$$\Delta \G_{(I , \lambda )}^{(r)}  = \sum_{I_1 ; \, I=I_1 \cdot I_2} \,\, \sum_{K \subset \lambda} \,\G_{\mathrm{Std}((I_1 , K))}^{(r)} \otimes \G_{\mathrm{Std}((I_2 ,(\lambda \backslash K)))}^{(r)} ,$$

\noindent where the standardization $\mathrm{Std}(m)$ of $m\in \mathrm{BW}$ is the standardization of the first row of the bi-word $m$.
\end{prop}

\begin{proof}

Let $(I,\lambda)$ be a $r$-bi-word with $n$ non-zero columns. Using formula \eqref{formR}, the linearity and the definition of the coproduct in $\PQSymstar$ give
$$ \Delta \G_{(I , \lambda )}^r = \Delta \left(\sum_{C \in I \,\shuffle \, \lambda^{\sigma} } \G_C \right) = \sum_{C \in I \,\shuffle \, \lambda^{\sigma} } \Delta \G_C  .$$

Using bijection $\phi$ of Proposition \ref{bijectionphi} to reformulate Proposition $3.5$ of \cite{novelli2005construction}, we can write the coproduct

$$\Delta \G_C = \sum_{i=0}^{n} \, \G_{\mathrm{Std}([C_1, ..., C_i])} \otimes \G_{\mathrm{Std}([C_{i+1}, ..., C_n])} , $$
 
 \noindent where $C=[C_1,C_2,...,C_n]$ with $C_i$ the $i^{th}$ non-zero column of $C$. Thus

 $$\Delta \G_{(I , \lambda )}^r = \sum_{C \in I \,\shuffle \, \lambda^{\sigma} } \,\sum_{i=0}^{n} \, \G_{\mathrm{Std}([C_1, ..., C_i])} \otimes \G_{\mathrm{Std}([C_{i+1}, ..., C_n])} .$$
 
The set $\{ \,([C_1,...,C_i] , [C_{i+1},...,C_n]),\, \text{for} \, i\in \{0,...,n\} \, \text{and} \, C\in I \shuffle \lambda^{\sigma} \} $ is equal to the set
$\{ (J,J'),\, \text{for}\, J\in I_1 \shuffle K^{\sigma} \, \text{and} \, J'\in I_2 \shuffle (\lambda \backslash K)^{\sigma} \, \text{where} \, I=I_1 \cdot I_2 \, \text{and} \, K\subset \lambda  \}$. Using this equality, we have

\begin{align*}
\Delta \G_{(I , \lambda )}^r  &= \sum_{I_1 ; \, I=I_1 \cdot I_2} \,\, \sum_{K \subset \lambda}\, \sum_{J\in I_1 \shuffle K^{\sigma}} \, \sum_{J'\in I_2 \shuffle (\lambda \backslash K)^{\sigma}} \, \G_{\mathrm{Std}(J)} \otimes \G_{\mathrm{Std}(J')} \\
 &= \sum_{I_1 ; \, I=I_1 \cdot I_2} \,\, \sum_{K \subset \lambda}\, \left( \sum_{J\in I_1 \shuffle K^{\sigma}} \,\G_{\mathrm{Std}(J)} \right) \otimes \left( \sum_{J'\in I_2 \shuffle (\lambda \backslash K)^{\sigma}} \, \G_{\mathrm{Std}(J')} \right) \\
  &= \sum_{I_1 ; \, I=I_1 \cdot I_2} \,\, \sum_{K \subset \lambda} \,\G_{\mathrm{Std}((I_1 , K))}^r \otimes \G_{\mathrm{Std}((I_2 ,(\lambda \backslash K)))}^r .
\end{align*}
\end{proof}

To prove that ${\PQSymstar}^r$ is a subalgebra of $\PQSymstar$, we first give the following easy result. Although it is quite long to write, it is just to understand the concatenation of two words at the level of their associated bi-words (see Example \ref{exconcpro}). 

\begin{lem}
\label{lemprodconc}
Let $m\in \mathrm{BW}_n$ with first row given by the $L_i$ for $i\geq 1$. Let $k\leq n$. As before, we identify $m$ with the word on $\mathbb{N}\setminus \{0\}$ that bijectively maps to $m$ using $\phi$. We write $m=m_1\cdot m_2$ where the length of $m_1$ is $k$. Then the bi-word $m_1$ is obtained with the following algorithm :
\begin{itemize}
    \item[$\bullet$] First form a finite bi-word $m'$ by scanning from left to right the non-empty columns of $m$ that contain in the first row letters that are less or equal to $k$, keep only these letters and keep the letters that are below them on the second row but incremented by $\displaystyle \sum_{p=1}^{i-1} |L_p|$ if they are in the $i^{th}$ column of $m$.
    \item[$\bullet$] In general $m'$ will not be in $\mathrm{BW}$, as the words of the second row will not always be prime parking functions. The way to make $m'$ an element of $\mathrm{BW}$ is to add empty columns. Each empty column that is at the left of a non-empty column $C_i$ of $m'$ will decrease by one all the letters of the word of the second row of $C_i$. There is a unique choice of adding empty columns to make these latter words prime parking functions, and it may split some of the columns of $m'$.
\end{itemize} 

We can obtain $m_2$ with the same algorithm but at the beginning we keep only the letters that are bigger than $k$, then decrease all these letters by $k$.
\end{lem}

\begin{ex}
\label{exconcpro}
    Let $m= \begin{tabular}{|cccccc|}
    $479$ & $28$ & $\epsilon$ & $1356$ & $\epsilon$ & $\cdots$  \\
    $211$ & $11$ & $\epsilon$ & $1131$ & $\epsilon$ & $\cdots$ 
\end{tabular} = 747297141$. Let $k=4$. We have $m=m_1\cdot m_2=7472 \cdot 97141$. We apply Lemma \ref{lemprodconc} to give directly $m_1$ and $m_2$ as bi-words using the bi-word $m$. For $m_1$ the first step gives us $m'= \begin{tabular}{|ccc|}
    $4$ & $2$ & $13$ \\
    $2$ & $4$ & $77$
\end{tabular}$. The second gives us $m_1 = \begin{tabular}{|ccccccccc|}
    $\epsilon$ & $4$ & $\epsilon$ & $2$ & $\epsilon$ & $\epsilon$ & $13$ & $\epsilon$ & $\cdots$  \\
    $\epsilon$ & $1$ & $\epsilon$ & $1$ & $\epsilon$ & $\epsilon$ & $11$ & $\epsilon$ & $\cdots$ 
\end{tabular}$. For $m_2$, we first obtain $\begin{tabular}{|ccc|}
    $79$ & $8$ & $56$ \\
    $11$ & $4$ & $97$
\end{tabular}$, and in fact by decreasing by $4$ the letters of the first row we obtain $\begin{tabular}{|ccc|}
    $35$ & $4$ & $12$ \\
    $11$ & $4$ & $97$
\end{tabular}$. Finally $m_2=\begin{tabular}{|cccccccccc|}
    $35$ & $\epsilon$ & $4$ & $\epsilon$ & $\epsilon$ & $2$ & $\epsilon$ & $1$ &$\epsilon$ &  $\cdots$   \\
    $11$ & $\epsilon$ & $1$ & $\epsilon$ & $\epsilon$ & $1$ & $\epsilon$  & $1$ &$\epsilon$ &  $ \cdots$
\end{tabular}$.
\end{ex}

\begin{prop}
    The subspace ${\PQSymstar}^r$ is a subalgebra of $\PQSymstar$. We have
    $$\G_u^r . \G_v^r = \sum \limits_{\substack{w\in \mathrm{PF} \,;\, w=a\cdot b \\ a\in \mathrm{Orb}^r(u),\,b\in \mathrm{Orb}^r(v)}} \frac{1}{|\mathrm{Orb}^r(w)\cap \mathrm{PF} |} \G_w^r \,.$$
\end{prop}

\begin{proof}
 We denote $u:=(I,\lambda)$ and $v:=(I',\lambda')$. Thanks to Proposition \ref{propreduce}, we have 
\begin{align*}
 \G_{(I,\lambda)}^r \G_{(I',\lambda')}^r &=\sum \limits_{\substack{K \in I\,\shuffle \, \lambda^{\sigma} \\ K' \in I'\,\shuffle \, {\lambda'}^{\sigma}}} \G_K \G_{K'}  \\
 &= \sum \limits_{\substack{K \in I\,\shuffle \, \lambda^{\sigma} \\ K' \in I'\,\shuffle \, {\lambda'}^{\sigma}}}  \sum \limits_{\substack{w\in \mathrm{PF} \,;\, w=a\cdot b \\ \mathrm{Park}(a)=K,\,\mathrm{Park}(b)=K'}} \G_w \\
 \G_{(I,\lambda)}^r \G_{(I',\lambda')}^r &= \sum \limits_{\substack{w\in \mathrm{PF} \,;\, w=a\cdot b \\ a\in \mathrm{Orb}^r(u),\,b\in \mathrm{Orb}^r(v)}} \G_w \,.
\end{align*}

Proposition \ref{propreduce} shows us that the subspaces ${\PQSymstar}^r$ have a basis whose vectors are disjoint sums of $\G_u$. It then remains to prove that the $\G_u$ with $u$ in the same $r$-orbit appear with the same coefficient, which is in fact always $0$ or $1$. It is a consequence of Lemma \ref{lemprodconc}, this Lemma shows that the columns of length smaller than $r$ in the concatenation of two words come from such columns in both smaller words. 
\end{proof}

\begin{ex}
\begin{align*}
\G_{123}^{(2)} \G_1^{(2)} &= \G_{2131}^{(2)} + \G_{1421}^{(2)} + \G_{3411}^{(2)} + \G_{1234}^{(2)} + \G_{2141}^{(2)} + \G_{1341}^{(2)} + \G_{3141}^{(2)} 
+ \G_{1231}^{(2)} + \G_{4121}^{(2)} \\ 
&+ \G_{1241}^{(2)} + \G_{1321}^{(2)} + \G_{4211}^{(2)} + \G_{2311}^{(2)} + \G_{2411}^{(2)} + \G_{3211}^{(2)} + \G_{3121}^{(2)} .\\
\Delta \G_{1133467}^{(2)} &= 1\otimes \G_{1133467}^{(2)} + 2\,\, \G_{1}^{(2)} \otimes \G_{113346}^{(2)} + \G_{11}^{(2)} \otimes \G_{11245}^{(2)} + \G_{12}^{(2)} \otimes \G_{11334}^{(2)} \\
&+ 2\,\, \G_{113}^{(2)} \otimes \G_{1124}^{(2)} 
+ \G_{1134}^{(2)} \otimes \G_{112}^{(2)} + \G_{11334}^{(2)} \otimes \G_{12}^{(2)} 
 + 2\,\, \G_{113346}^{(2)} \otimes \G_{1}^{(2)} \\
 &+ \G_{1133467}^{(2)} \otimes 1  .
\end{align*}
\end{ex}

We deduce

\begin{thm}
\label{thminterpolation}
We have an infinite chain of noncommutative and noncocommutative (except ${\PQSymstar}^{\infty}$ which is cocommutative) nested graded Hopf subalgebras
$$\PQSymstar = {\PQSymstar}^1  \supseteq {\PQSymstar}^2 \supseteq \dots  \supseteq {\PQSymstar}^r \supseteq \dots  \supseteq {\PQSymstar}^{\infty} .$$
\end{thm}

\section{The case $r=\infty$}
\label{sectioncasinfini}

Here we make a connection between the algebra ${\PQSymstar}^{\infty}$ and certain particular trees, giving us a new point of view on enumerative results obtained previously on these trees.

Let $T$ be a rooted labeled tree. The\defn{maximal decreasing subtree} of $T$, denoted by $MD(T)$, is the maximal subtree of $T$ which contains the root and whose edges are decreasing. For instance,

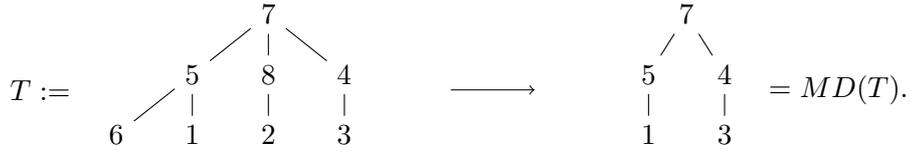
\begin{tikzpicture}
\begin{scope}[scale= 1]

\draw (-3,-1) node{$T:=$};

\node (A) at (0,0) {7};
\node (B) at (-1,-0.8) {5};
\node (C) at (0,-0.8) {8};
\node (E) at (1,-0.8) {4};
\node (F) at (-2,-1.6) {6};
\node (G) at (-1,-1.6) {1};
\node (D) at (0,-1.6) {2};
\node (I) at (1,-1.6) {3};

\draw (A) to (B);
\draw (A) to (C);
\draw (A) to (E);
\draw (B) to (F);
\draw (B) to (G);
\draw (C) to (D);
\draw (E) to (I);

\draw[->] (2.4,-1) -- (3.5,-1);

\begin{scope}[xshift= 5.5cm]

\draw (2,-1) node{$=MD(T).$};

\node (A) at (0,0) {7};
\node (B) at (-0.5,-0.8) {5};
\node (E) at (0.5,-0.8) {4};
\node (G) at (-0.5,-1.6) {1};
\node (I) at (0.5,-1.6) {3};

\draw (A) to (B);
\draw (A) to (E);
\draw (B) to (G);
\draw (E) to (I);
\end{scope}
\end{scope}
\end{tikzpicture}

\begin{prop}[Seo-Shin~\cite{seo2012enumeration}, Theorem 1]
\label{numberdecreasingsubtree}
Let $n\in \mathbb{N}$ and $0\leq k\leq n$ and $S(n,k)$ be the Stirling numbers of the second kind. Let $T_{n,k}$ be the set of rooted labeled non-planar trees on $[n+1]$ whose maximal decreasing subtree has $k+1$ vertices. Then
\begin{equation}
\label{eq:numberdecreasingsubtree}
|T_{n,k}| = k! \sum_{m=k}^n \binom{n+1}{m+1} S(m+1,k+1) (n-k)^{n-m-1} (m-k) .
\end{equation}
\end{prop}

A\defn{minimal rooted (labeled non-planar) tree} is a rooted labeled non-planar tree whose maximal decreasing subtree has just one vertex, the root of the tree. 

In Section $3$ of \cite{Foata1973MappingsOA}, Foata and Riordan give a bijection between $PF_n$ and the set of rooted labeled non-planar forests on $[n]$.

\begin{prop}[Foata-Riordan~\cite{Foata1973MappingsOA}, Section 3]
\label{bijectionFoata-Riordan}
The following map is a bijection

\begin{center}
\begin{tabular}{cccc}
    $\phi(u):$ & $PF_n$ & $\rightarrow$ & $\{f:[n]\rightarrow [n],\, f \,\text{acyclic}\}$ \\
       & $u$ & $\mapsto$ &  $\left\{\begin{tabular}{cc}
    $i$ & \text{if}$\,\,u_i=1 ,$ \\
    $\mathrm{Std}(u)^{-1}(u_i-1)$ & \text{otherwise}.
\end{tabular} \right.$
\end{tabular}
\end{center}
\end{prop}

\noindent Thanks to this bijection, Rattan obtained the following result.

\begin{prop}[Rattan~\cite{Rattan2006PermutationFA}, Section 4.2]
\label{corRattan}
There is a bijection between prime parking functions of length $n$ and minimal rooted labeled trees on $[n]$.
\end{prop}

\begin{prop}
\label{propbijectionarbreschaine}
The set $M_{n,k}$ of sets of $k$ minimally rooted trees such that the vertex sets of the trees form a set partition of $[n]$ and the set $T'_{n,k}$ of rooted labeled non-planar trees on $[n]$ whose maximal decreasing subtrees is a chain with $k$ vertices are in bijection.

\end{prop}

\begin{proof}
Let $\{T_1,T_2,\dots ,T_k\}\in M_{n,k}$ be a forest ordered in such a way that $root(T_1)>\dots > root(T_k)$ where $root(T_i)$ is the label of the root of $T_i$. We obtain an element of $T'_{n,k}$ by adding edges connecting the roots of $T_i$ and $T_{i+1}$ for all $1\leq i\leq k-1$ and making the root of $T_1$ be the root of the resulting tree. This gives a bijection between $M_{n,k}$ and $T'_{n,k}$.
\end{proof}

We have $\displaystyle |T'_{n+1,k+1}|=\frac{|T_{n,k}|}{k!}$ since in Formula \eqref{eq:numberdecreasingsubtree}, $k!$  counts the number of decreasing trees on $k+1$ vertices.

\begin{prop}
\label{equalityAandT}
We have, for $n\geq 1$ and $1\leq k\leq n$, a bijection between the vectors $\G_u^{(\infty)}$, where $u$ is an $\infty$-bi-word of length $n$ with $k$ non-empty columns, and the set $T'_{n,k}$. Then
\begin{align*}
A_{n+1,k+1}^{\infty} = \frac{|T_{n,k}|}{k!} = \sum_{m=k}^{n} \binom{n}{m} S(m,k) \,(n-k)^{n-m} = \sum_{i=0}^{k} (-1)^i \binom{k}{i} (n-i)^{n}.    
\end{align*}
\end{prop}

\begin{proof}
Let $u=\begin{tabular}{|cccc|}
    $L_1$ & $L_2$ & $\cdots$ & $L_k$\\
    $M_1$ & $M_2$ & $\cdots$ & $M_k$
\end{tabular}$ be an $\infty$-bi-word of length $n$ with $k$ non-empty columns. Using Proposition~\ref{corRattan}, we send all prime parking functions $M_i$ to minimally rooted trees $T_i$. We use the letters of $L_i$ to relabel the vertices of $T_i$, the vertex of label $j$ in $T_i$ being relabeled by the $j^{th}$ letter of $L_i$. Thus we obtain a set of $k$ minimally rooted trees whose vertex sets form a set partition of $[n]$, that is an element of $M_{n,k}$.
We then conclude using Proposition~\ref{propbijectionarbreschaine}, mapping this set to a tree whose maximal decreasing subtree is a chain with $k$ vertices. Then the first equality is a direct consequence of this bijection since $A_{n,k}^{\infty}$ counts the number of the vectors $\G_u^{(\infty)}$ where $u$ is an $\infty$-bi-word of length $n$ with $k$ non-empty columns. The second and third equalities are Theorem $3$ of~\cite{seo2012enumeration}.
\end{proof}

Thus, we have obtained a bijective proof of the following.

\begin{thm}
\label{thdimtrees}
For any $n\geq 1$, $\mathrm{dim}({\textbf{PQSym}_n^*}^{\infty}) = \sum_{k=1}^n A_{n,k}^{\,\infty}$ is the number of rooted labeled non-planar trees on $[n]$ whose maximal decreasing subtree is a chain.
\end{thm}

\begin{remark}
For $r=1$, Propositions \ref{propdimensionpqsymr} and \ref{equalityAandT} give us back Equation $(12)$ of~\cite{seo2012enumeration}:
$$(n+2)^n = dim(\PQSymstar_{n+1}) = \sum_{k=0}^{n} (k+1)! \,A_{n+1,k+1}^{\infty} = \sum_{k=0}^{n} (k+1) |T_{n,k}| .$$
\end{remark}

\begin{remark}
There is a bijection between the set $\hat{\mathcal{T}}_n$ of minimal rooted trees on $[n]$ whose root $x_1$ is smaller than all the children of its smallest child $v_1$ and the set $M_{n,2}$.

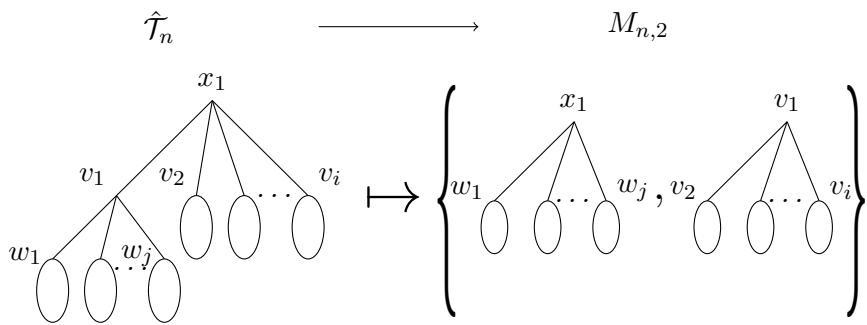
\begin{tikzpicture}
\begin{scope}[scale= 0.7]
\draw (2,0) node{$\hat{\mathcal{T}}_n$};
\draw[->] (5,0) -- (8,0);
\draw (11,0) node{$M_{n,2}$};

\begin{scope}[yshift= -5cm, scale= 0.3]
\draw (0,0) ellipse (1cm and 2cm);
\draw (3,0) ellipse (1cm and 2cm);
\draw (7,0) ellipse (1cm and 2cm);
\draw (9,4) ellipse (1cm and 2cm);
\draw (12,4) ellipse (1cm and 2cm);
\draw (16,4) ellipse (1cm and 2cm);

\draw (4,6) -- (10,12);
\draw (9,6) -- (10,12);
\draw (12,6) -- (10,12);
\draw (16,6) -- (10,12);

\draw (4,6) -- (0,2);
\draw (4,6) -- (3,2);
\draw (4,6) -- (7,2);

\draw (10,12) node[above]{\small{$x_1$}};
\draw (4,6) node[above left]{\small{$v_1$}};
\draw (9,6) node[above left]{\small{$v_2$}};
\draw (16,6) node[above right]{\small{$v_i$}};

\draw (0,2.2) node[left]{\small{$w_1$}};
\draw (7,2.2) node[left]{\small{$w_j$}};

\draw (5,1.4) node{\small{\dots}};
\draw (14,6) node{\small{\dots}};
\end{scope}

\draw (6.4,-3.3) node[scale= 2]{$\mapsto$};

\begin{scope}[xshift= 7.8cm, yshift= -3.3cm]
\node at (-0.4,0) [xscale= 1.7, yscale= 8]{$\{$};
\draw (0.5,0) -- (2,1.5);
\draw (1.5,0) -- (2,1.5);
\draw (2.6,0) -- (2,1.5);

\draw (2,1.5) node[above]{\small{$x_1$}};
\draw (0.5,0.2) node[left]{\small{$w_1$}};
\draw (2.6,0.2) node[right]{\small{$w_j$}};

\draw (0.5,-0.5) ellipse (0.25cm and 0.5cm);
\draw (1.5,-0.5) ellipse (0.25cm and 0.5cm);
\draw (2.6,-0.5) ellipse (0.25cm and 0.5cm);
\draw (2,0) node{\small{\dots}};

\begin{scope}[xshift= 4cm]
\node at (3.3,0) [xscale= 1.7, yscale= 8]{$\}$};
\draw (-0.4,0) node[scale= 1.5]{$,$};
\draw (0.5,0) -- (2,1.5);
\draw (1.5,0) -- (2,1.5);
\draw (2.6,0) -- (2,1.5);

\draw (2,1.5) node[above]{\small{$v_1$}};
\draw (0.5,0.2) node[left]{\small{$v_2$}};
\draw (2.6,0.2) node[right]{\small{$v_i$}};

\draw (0.5,-0.5) ellipse (0.25cm and 0.5cm);
\draw (1.5,-0.5) ellipse (0.25cm and 0.5cm);
\draw (2.6,-0.5) ellipse (0.25cm and 0.5cm);
\draw (2,0) node{\small{\dots}};
\end{scope}
\end{scope}
\end{scope}
\end{tikzpicture}

\vspace{-10mm}
Using Proposition~\ref{propbijectionarbreschaine}, it gives a bijective proof of Corollary $4.4$ of~\cite{Rattan2006PermutationFA} that says that the number of prime parking functions of length $n$ such that no letters $2$ appear to the left of the left most $1$, which is in bijection with $\hat{\mathcal{T}}_n$ by looking at the proof of Proposition~\ref{corRattan}, is equal to $A_{n,2}^{\infty}$ and thus  $(n-1)^{n-1} - (n-2)^{n-1}$ by Proposition~\ref{equalityAandT}.
\end{remark}

\section*{Acknowledgements}

I am very thankful to Samuele Giraudo and Jean-Christophe Novelli for helping me throughout this project. I also want to thank Hugh Thomas.

\printbibliography
\end{document}